\documentclass[12pt,a4paper,]{amsart}
\usepackage{amsthm,amsmath,amsfonts,amssymb}
\usepackage{hyperref}
\usepackage{color}
\usepackage{mathtools}
\usepackage{amsmath,amsfonts}

\renewcommand{\proof}{{\noindent \bf Proof:\ }}

\newtheorem{theorem}{Theorem}[section]
\newtheorem{lemma}[theorem]{Lemma}
\newtheorem{corollary}[theorem]{Corollary}
\newtheorem{definition}[theorem]{Definition}
\newtheorem{proposition}[theorem]{Proposition}

\newtheorem{remark}{Remark}[section]

\title[Pullback attractors for a class of non-autonomous systems]{Pullback attractors for a class of non-autonomous thermoelastic plate systems}

\numberwithin{equation}{section} \numberwithin{theorem}{section}

\author[F. D. M. Bezerra]{Flank D. M. Bezerra$^\dag$}\thanks{$^\dag$Research partially
supported by FAPESP \#2014/03686-3, Brazil}
\address[F. D. M. Bezerra]{Universidade Federal da Para\'{\i}ba, Departamento de Matem\'atica, 58051-900 Jo\~{a}o Pessoa PB, Brazil.}
\email{flank@mat.ufpb.br}

\author[V. L. Carbone]{Vera L. Carbone}
\address[V. L. Carbone]{Universidade Federal de S\~{a}o
Carlos, Departamento de Matem\'atica, 13565-905 S\~{a}o
Carlos SP, Brazil.}
\email{carbone@dm.ufscar.br}

\author[M. J. D. Nascimento]{Marcelo J. D. Nascimento$^\star$}\thanks{$^\star$Research partially
supported by FAPESP \#2014/03109-6, Brazil}
\address[M. J. D. Nascimento]{Universidade Federal de S\~{a}o
Carlos, Departamento de Matem\'atica, 13565-905 S\~{a}o
Carlos SP, Brazil.}
\email{marcelo@dm.ufscar.br}

\author[K. Schiabel]{Karina Schiabel}
\address[K. Schiabel]{Universidade Federal de S\~{a}o
Carlos, Departamento de Matem\'atica, 13565-905 S\~{a}o
Carlos SP, Brazil.}
\email{schiabel@dm.ufscar.br}

\date{\today}

\begin{document}

\maketitle

\begin{abstract}
In this article we study the asymptotic behavior of solutions, in sense of global pullback attractors, of the evolution system
\[
\begin{cases}
u_{tt} +\eta\Delta^2 u+a(t)\Delta\theta=f(t,u), & t>\tau,\ x\in\Omega,\\
\theta_t-\kappa\Delta \theta-a(t)\Delta u_t=0, & t>\tau,\ x\in\Omega,
\end{cases}
\]
subject to boundary conditions
\[
u=\Delta u=\theta=0,\ t>\tau,\ x\in\partial\Omega,
\]
where $\Omega$ is a bounded domain in $\mathbb{R}^N$ with $N\geqslant 2$, which boundary $\partial\Omega$ is assumed to be a $\mathcal{C}^4$-hypersurface, $\eta>0$ and $\kappa>0$ are constants, $a$ is an H\"older continuous function, $f$ is a dissipative nonlinearity locally Lipschitz in the second variable. 

\end{abstract}

\section{Introduction}

This paper is concerned with the qualitative behavior of non-autonomous dynamical systems generated by a non-autonomous thermoelastic plate system, in particular as described by their pullback attractors. The problem that we consider in this paper is the following. Let $\Omega$ be a bounded domain in $\mathbb{R}^N$ with $N\geqslant 2$, which boundary $\partial\Omega$ is assumed to be a $\mathcal{C}^4$-hypersurface. We consider the following initial-boundary value problem
\begin{equation}\label{prob1}
\begin{cases}
u_{tt} +\eta\Delta^2 u+a(t)\Delta\theta=f(t,u),\ & t>\tau,\ x\in\Omega,\\
\theta_t-\kappa\Delta \theta-a(t)\Delta u_t=0,\ & t>\tau,\ x\in\Omega,
\end{cases}
\end{equation}
where $\eta$ and $\kappa$ are positive constants, subject to boundary conditions
\begin{equation*}\label{prob2}
\begin{cases}
u=\Delta u=0,\ t>\tau,\ x\in\partial\Omega,\\
\theta=0,\ t>\tau,\ x\in\partial\Omega,
\end{cases}
\end{equation*}
and initial conditions
\begin{equation}\label{prob3}
u(\tau,x)=u_0(x),\ u_t(\tau,x)=v_0(x)\ \mbox{and}\ \theta(\tau,x)=\theta_0(x),\ x\in\Omega,\ \tau\in\mathbb{R}.
\end{equation}

Here, we assume that there exist positive constants $a_0$ and $a_1$ such that
\begin{equation}\label{eta-c1}
0<a_0\leqslant a(t)\leqslant a_1,\ \forall t\in\mathbb{R}.
\end{equation}
Furthermore, we assume that the function $a$ is $(\beta,C)-$H\"{older} continuous; that is,
\begin{equation}\label{eta-c2}
|a(t)-a(s)|\leqslant C|t-s|^\beta,\ \forall t,s\in\mathbb{R}.
\end{equation}

Below we give conditions under which the non-autonomous problem \eqref{prob1}-\eqref{prob3} is locally and globally well posed in some space that we will specify later. To that end we must assume some growth condition on the nonlinearity $f$.

To obtain the global existence of solutions and the existence of pullback attractor we assume that $f:\mathbb{R}^2\to\mathbb{R}$ is locally Lipschitz in the second variable, and it is a dissipative nonlinearity in the second variable
\begin{equation}\label{Cond_f1}
\limsup_{|s|\to\infty}\dfrac{f(t,s)}{s}<\lambda_1, 
\end{equation}
uniformly in $t\in\mathbb{R}$, where $\lambda_1>0$ is the first eigenvalue of negative Laplacian operator with zero Dirichlet boundary condition. Due to Sobolev embedding we need to assume that the function $f$ satisfies subcritical growth condition; that is,
\begin{equation}\label{Cond_f2}
|f_s(t,s)|\leqslant C(1+|s|^{\rho-1}),\ \forall s\in\mathbb{R},
\end{equation}
where $1\leqslant\rho<\frac{N}{N-4}$, with $N\geqslant 5$, and $C>0$ independent of $t\in\mathbb{R}$. We will justify these restrictions later in the paper. If $N=2,3,4$, we suppose the growth condition \eqref{Cond_f2} with $\rho\geqslant1$.

Although we get global solutions for initial-boundary value problem \eqref{prob1}-\eqref{prob3} in the space $H^2(\Omega)\times L^2(\Omega)\times L^2(\Omega)$ with the nonlinearity $f$ satisfying the growth condition \eqref{Cond_f2} with $1\leqslant\rho<\frac{N+4}{N-4}$ for $N\geqslant 5$, see Section \ref{Sec-Exist}, to study the asymptotic behavior of solutions we only consider the condition \eqref{Cond_f2} with $1\leqslant\rho<\frac{N}{N-4}$ for $N\geqslant 5$, see Section \ref{Sec-Diss}.


The purpose of this paper is to prove, under suitable assumptions, local and global well-posedness (using the uniform sectorial operators theory) of the non-autonomous problem \eqref{prob1}-\eqref{prob3}, the existence of pullback attractors and uniform bounds for these pullback attractors.

It is very well known that the model in \eqref{prob1} describes the small vibrations of a homogeneous, elastic and thermal isotropic Kirchhoff plate. In the literature the initial boundary-value problem \eqref{prob1}-\eqref{prob3} has been extensively discussed for several authors in differents contexts. For instance, Baroun et al. in \cite{BBDM} studied the existence of almost periodic solutions for an evolution system like \eqref{prob1}, Liu and Renardy \cite{LR} proved that the linear semigroup defined by system \eqref{prob1} with $f\equiv0$ with clamped boundary condition for $u$ and Dirichlet boundary condition for $\theta$ is analytic. The typical difficulties in thermoelasticity comes from the boundary condition, which make more complicated the task of getting estimates to show the exponential stability of the solutions or analyticity of the corresponding semigroup. In that direction we have the works of Liu and Zheng \cite{LZ}, Lasiecka and Triggiani \cite{LT} to free - clamped boundary condition. In this last work the authors show the exponential stability and analyticity of the semigroup associated with the system \eqref{prob1}. We refer to the book of Liu and Zheng \cite{LZ1} for a general survey on those topics.

To formulate the non-autonomous problem \eqref{prob1}-\eqref{prob3} in the nonlinear evolution process setting, we introduce some notations. Here, we denote $X=L^2(\Omega)$ and $\Lambda:D(\Lambda)\subset X\to X$ the biharmonic operator defined by $D(\Lambda)=\{u\in H^4(\Omega);\ u=\Delta u=0\ \mbox{on}\ \partial\Omega\}$ and 
\[
\Lambda u=(-\Delta)^2 u,\ \forall u\in D(\Lambda),
\]
then $\Lambda$ is a positive self-adjoint operator in $X$ with compact resolvent and therefore $-\Lambda$ generates a compact analytic semigroup on $X$ (that is, $\Lambda$ is a sectorial operator, in the sense of Henry \cite{H}). Denote by $X^\alpha$, $\alpha > 0$,  the fractional power spaces associated with the operator $\Lambda$; that is, $X^\alpha=D(\Lambda^\alpha)$ endowed with the graph norm.  With this notation, we have $X^{-\alpha}=(X^\alpha)'$ for all $\alpha>0$, see Amann \cite{A} for the characterization of the negative scale. It is of special interest the case $\alpha=\frac12$, since $-\Lambda^{\frac12}$ is the Laplacian operator with homogeneous Dirichlet boundary conditions.

If we denote $v=u_t$, then we can rewrite the non-autonomous problem \eqref{prob1}-\eqref{prob3} in the matrix form 
\begin{equation}\label{Def-Prob}
\begin{cases}
w_t=A_{(a)}(t)w+F(t,w),\ t>\tau,\\
w(\tau)=w_0,\ \tau\in\mathbb{R},
\end{cases} 
\end{equation}
where $w=w(t)$ for all $t\in\mathbb{R}$, and $w_0=w(\tau)$ are given by 
\begin{equation}\label{Def1}
w=\begin{bmatrix} u\\ v\\ \theta\end{bmatrix},\ \mbox{and}\ w_0=\begin{bmatrix} u_0\\ v_0\\ \theta_0\end{bmatrix},
\end{equation}
and, for each $t\in\mathbb{R}$, the unbounded linear operator $A_{(a)}(t):D(A_{(a)}(t))\subset Y\to Y$ is defined by 
\begin{equation*}\label{Def2}
Y=H^2(\Omega)\times L^2(\Omega)\times L^2(\Omega),
\end{equation*}
being phase space of the problem \eqref{prob1}-\eqref{prob3}; the domain of the operator $A_{(a)}(t)$ is defined by the space
\begin{equation}\label{Def3}
D(A_{(a)}(t))=\{u\in H^4(\Omega);\ u=\Delta u=0\ \mbox{on}\ \partial\Omega \}\times H^2(\Omega)\times (H^2(\Omega)\cap H^1_0(\Omega))
\end{equation}
and 
\begin{equation}\label{Def4}
\begin{split}
A_{(a)}(t)\begin{bmatrix} u\\ v\\ \theta\end{bmatrix}&=\begin{bmatrix} 0 & I & 0\\ -\eta\Lambda & 0 & -a(t)\Lambda^\frac{1}{2}\\ 0 & a(t)\Lambda^\frac{1}{2} & \kappa\Lambda^\frac{1}{2} \end{bmatrix}\begin{bmatrix} u\\ v\\ \theta\end{bmatrix}\\
&=\begin{bmatrix} v\\ -\eta\Lambda u - a(t)\Lambda^\frac{1}{2} \theta\\ a(t)\Lambda^\frac{1}{2} v+\kappa\Lambda^\frac{1}{2}\theta\end{bmatrix},
\end{split}
\end{equation}
for all $t\in\mathbb{R}$.

We define the nonlinearity $F$ by
\begin{equation*}\label{Def-F}
F(t,w)=\begin{bmatrix} 0\\ f^e(t,u)\\ 0\end{bmatrix},
\end{equation*}
where $f^e(t,u)$ is  the Nemitski\u{\i} operator associated with $f(t,u)$, $t\in\mathbb{R}$, that is,
\[
f^e(t,u)(x):=f(t,u(x)),\ \forall t\in\mathbb{R}, x\in\Omega.
\]
The map $f^e(t,u)$ is Lipschitz continuous in bounded subsets of $X^{\frac12}$ uniformly in $t\in\mathbb{R}$.

This paper is organized as: In Section 2 we recall concepts and results about problems singularly non-autonomous, including results on existence of pullback attractors. In Section 3 we deal with the linear problem associated \eqref{prob1}-\eqref{prob3}. Section 4 is devoted to study the existence of local and global solutions in some appropriate space. Finally, in Section 5 we present results on dissipativeness of thermoelastic equation and existence of pullback attractors for \eqref{prob1}-\eqref{prob3}.

\section{Singularly non-autonomous abstract problem}

Throughout the paper, $L(\mathcal{Z})$ will denote the space of linear and bounded operators defined in a Banach space $\mathcal{Z}$. Let $\mathcal{A}(t)$, $t\in \mathbb{R}$, be a family of unbounded closed linear operators defined on a fixed dense subspace $D$ of $\mathcal{Z}$. 

\subsection{Singularly non-autonomous abstract linear problem}

Consider the singularly non-autonomous abstract linear parabolic problem of the form
\begin{equation*}
\begin{cases}
\displaystyle \frac{du}{dt} =- \mathcal{A}(t)u,\ t>\tau, \\ 
u(\tau) =u_0\in  D.
\end{cases}
\end{equation*}

We assume that

\begin{enumerate}
\item[$(a)$] The operator $\mathcal{A}(t):D\subset \mathcal{Z}\to \mathcal{Z}$ is a closed densely defined operator (the domain $D$ is fixed) and there is a constant $C>0$ (independent of $t\in \mathbb{R}$) such that
\begin{equation*}
\|(\mathcal{A}(t) + \lambda I)^{-1}\|_{L(\mathcal{Z})} \leqslant \frac{C}{|\lambda| +1}; \ \mbox{for all}\ \lambda\in \mathbb{C} \hbox{ with } {\rm Re}\lambda\geqslant 0.
\end{equation*}
To express this fact we will say that the family $\mathcal{A}(t)$ is \emph{uniformly sectorial}.
\item[$(b)$] There are constants $C>0$ and $\epsilon_0>0$ such that, for any
$t,\tau,s \in {\mathbb R}$,
\begin{equation*}
\|[\mathcal{A}(t)-\mathcal{A}(\tau)]\mathcal{A}^{-1}(s)\|_{L(\mathcal{Z})} \leqslant C(t-\tau)^{\epsilon_0}, \quad \epsilon_0 \in (0,1].
\end{equation*}
To express this fact we will say that the map $\mathbb{R}\ni t\mapsto\mathcal{A}(t)$ is \emph{uniformly H\"{o}lder continuous}.
\end{enumerate}

Denote by $\mathcal{A}_0$ the operator $\mathcal{A}(t_0)$ for some $t_0\in \mathbb{R}$ fixed. If
$\mathcal{Z}^\alpha$ denotes the domain of $\mathcal{A}_0^\alpha$, $\alpha> 0$,  with the graph norm and $\mathcal{Z}^0:=\mathcal{Z}$, denote by $\{\mathcal{Z}^\alpha;\alpha\geqslant  0\}$
the fractional power scale associated with $\mathcal{A}_0$ (see Henry \cite{H}).

From $(a)$, $-\mathcal{A}(t)$ is the generator of an analytic semigroup
$\{e^{-\tau \mathcal{A}(t)}\in L(\mathcal{Z}): \tau \geqslant 0\}$. Using this and the fact that $0\in \rho(\mathcal{A}(t))$, it follows that
\begin{equation*}
\|e^{-\tau \mathcal{A}(t)}\|_{L(\mathcal{Z})} \leqslant C,\ \tau\geqslant 0, \ t\in \mathbb{R},
\end{equation*}
and
\begin{equation*}
\|\mathcal{A}(t)e^{-\tau \mathcal{A}(t)}\|_{L(\mathcal{Z})} \leqslant C\tau^{-1}, \ \tau >0, \ t\in \mathbb{R}. 
\end{equation*}

It follows from $(b)$ that $\|\mathcal{A}(t)\mathcal{A}^{-1}(\tau)\|_{L(\mathcal{Z})} \leqslant C$, $\forall \, (t,\tau ) \in I$, for some $I \subset \mathbb{R}^{2}$ bounded. Also, the semigroup $e^{-\tau \mathcal{A}(t)}$ generated by $-\mathcal{A}(t)$ satisfies the following estimate
\begin{equation*}
\|e^{-\tau \mathcal{A}(t)}\|_{L(\mathcal{Z}^\beta,\mathcal{Z}^\alpha)} \leqslant M\tau^{\beta-\alpha},
\end{equation*}
where $0\leqslant \beta \leqslant \alpha <1+\epsilon_0$.

\medskip

Next we recall the definition of a linear evolution process associated with a family of operators $\{\mathcal{A}(t):t\in \mathbb{R}\}$.
\begin{definition}
A family $\{L(t,\tau): t\geqslant \tau\in {\mathbb R}\}\subset L(\mathcal{Z})$
satisfying
\begin{equation*}
\begin{split}
&1)\ L(\tau,\tau) = I,\\
&2)\ L(t,\sigma)L(\sigma,\tau)=L(t,\tau), \ \hbox{for any }\; t\geqslant\sigma\geqslant \tau, \\
&3)\ \mathcal{P}\times \mathcal{Z} \ni ((t,\tau),u_0)\mapsto L(t,\tau)v_0\in \mathcal{Z} \;\;\hbox{is continuous, where }
\mathcal{P}=\{(t,\tau)\in \mathbb{R}^2: t\geqslant \tau\}
\end{split}
\end{equation*}
is called a \emph{linear evolution process} (\emph{process} for short)
or \emph{family of evolution operators}.
\end{definition}

If the operator $\mathcal{A}(t)$ is uniformly sectorial and uniformly H\"{o}lder continuous, then  there exists a linear evolution process $\{L(t,\tau): t\geqslant \tau
\in \mathbb{R}\}$ associated with $\mathcal{A}(t)$, which is given by
\[
L(t,\tau) = e^{-(t-\tau)\mathcal{A}(\tau)} + \int_{\tau}^{t}L(t,s)[\mathcal{A}(\tau) -
\mathcal{A}(s)]e^{-(s-\tau)\mathcal{A}(\tau)}ds.
\]
The evolution process $\{U(t,\tau): t\geqslant \tau \in \mathbb{R}\}$ satisfies the following condition:
\begin{equation}\label{1.65}
\|L(t,\tau)\|_{\mathcal{L}(\mathcal{Z}^\beta,\mathcal{Z}^\alpha)} \leqslant C(\alpha,\beta)(t-\tau)^{\beta-\alpha},
\end{equation}
where $0\leqslant \beta\leqslant \alpha <1+\epsilon_0$. For more details see \cite{CN} and \cite{sobol}.

\subsection{Existence of pullback attractors}

In this subsection we will present basic definitions and results of the theory of pullback attractors for nonlinear evolution process. For more details we refer to \cite{CCLR}, \cite{TLR}, \cite{CLR} and \cite{ChV}.

We consider the singularly non-autonomous abstract parabolic problem
\begin{equation}\label{po}
\begin{cases}
\displaystyle \frac{du}{dt} =- \mathcal{A}(t)u+g(t,u),\ t>\tau, \\ 
u(\tau) =u_0\in  D,
\end{cases}
\end{equation}
where the operator $\mathcal{A}(t)$ is uniformly sectorial and uniformly H\"{o}lder continuous	and the nonlinearity $g$ satisfies conditions which will be specified later.  The nonlinear evolution process $\{S(t,\tau): t\geqslant \tau\in \mathbb{R}\}$ associated with $\mathcal{A}(t)$ is given by
\[
S(t,\tau) = L(t,\tau)+\int_\tau^t L(t,s)g(s,L(s,\tau))ds,\ \forall t\geqslant\tau.
\]

\begin{definition} 
Let $g: \mathbb{R}\times X^\alpha \to X^{\beta}$, $\alpha \in [\beta, \beta+1)$
be a continuous function. We say that a continuous function
$u:[\tau,\tau+t_0] \to X^\alpha $ is a $($\emph{local}$)$ \emph{solution}
of  \eqref{po} starting in $ u_0 \in X^\alpha$,
if $u \in C([\tau,\tau+t_0], X^\alpha) \cap C^1((\tau,\tau+t_0],
X^{\alpha})$, $u(\tau)=u_0$, $u(t) \in D(\mathcal{A}(t))$
for all $t \in (\tau,\tau+t_0]$ and \eqref{po}
is satisfied for all $t \in (\tau,\tau+t_0)$.
\end{definition}

We can now state the following result, proved in
\cite[Theorem 3.1]{CN}.

\begin{theorem}\label{teo:existunicsol}
Suppose that the family of operators $\mathcal{A}(t)$
is uniformly sectorial and uniformly H\"older continuous
in $X^{\beta}$. If $g:\mathbb{R}\times X^\alpha \to X^{\beta}$,
$\alpha \in [\beta, \beta +1)$, is a Lipschitz continuous
map in bounded subsets of $X^\alpha$, then, given $r > 0$,
there is a time $t_0 > 0$ such that for all
$u_0 \in B_{X^\alpha }(0;r)$ there exists a unique solution
of the problem \eqref{po} starting in $u_0$ and defined
in $[\tau,\tau+t_0]$. Moreover, such solutions are continuous
with respect the initial data in $B_{X^\alpha }(0;r)$.
\end{theorem}

We start remembering the definition of Hausdorff semi-distance
between two subsets $A$ and $B$ of a metric space $(X,d)$:
\[
\operatorname{dist}_H(A,B) = \sup_{a\in A} \inf_{b\in B} d(a,b).
\]

Next we present several definitions about theory of pullback attractors, which can be founded in \cite{CCLR, CLR, ChV}.

\begin{definition} 
Let $\{S(t,\tau):t\geqslant \tau\in {\mathbb R}\}$ be an evolution
process in a metric space $X$. Given  $A$ and $B$ subsets of $X$, we
say that $A$ \emph{pullback attracts} $B$ at time $t$ if
$$
\lim_{\tau \to -\infty} \operatorname{dist}_H(S(t,\tau)B,A)= 0,
$$
where $S(t,\tau)B:= \{S(t,\tau)x \in X : x \in B\}$.
\end{definition}

\begin{definition} 
The \emph{pullback orbit} of a subset $B \subset X$ relatively
to the evolution process
$\{S(t,\tau):t\geqslant \tau\in {\mathbb R}\}$ in the time
$t \in \mathbb{R}$ is defined by
$\gamma_p(B,t):= \cup_{\tau \leqslant t} S(t,\tau)B $.
\end{definition}

\begin{definition} 
An evolution process $\{S(t,\tau): t \geqslant \tau \in {\mathbb R}\}$ in $X$
is \emph{pullback strongly bounded} if, for each $t\in \mathbb{R}$ and
each bounded subset $B$ of $X$,
$\cup_{\tau\leqslant t} \gamma_p(B,\tau)$ is bounded.
\end{definition}

\begin{definition} 
An evolution process $\{S(t,\tau):t\geqslant \tau\in {\mathbb
R}\}$ in $X$ is \emph{pullback asymptotically compact} if, for each
$t \in \mathbb{R}$, each sequence $\{\tau_n\}$ in $(-\infty, t]$ with
$\tau_n\to -\infty$ as $n\to\infty$ and each
bounded sequence $\{x_n\}$ in $X$ such that
$\{S(t,\tau_n)x_n\} \subset X $ is bounded, the sequence
$\{S(t,\tau_n)x_n\}$ is relatively compact in $X$.
\end{definition}

\begin{definition} 
We say that a family of bounded subsets $\{B(t): t\in \mathbb{R}\}$ of
$X$ is \emph{pullback absorbing} for the evolution process
$\{S(t,\tau):t\geqslant \tau\in \mathbb{R}\}$, if for each $t\in \mathbb{R}$ and for
any bounded subset $B$ of $X$, there exists $\tau_0(t,B)\leqslant t$
such that
$$
S(t,\tau)B\subset B(t) \quad \mbox{ for all  }
 \tau \leqslant \tau_0(t,B).$$
\end{definition}

\begin{definition} 
A family of subsets $\{{\mathbb A}(t):t\in \mathbb{R}\}$ of $X$ is
 called a \emph{pullback attractor} for the evolution process
$\{S(t,\tau): t\geqslant \tau\in {\mathbb R}\}$ if it is invariant (that is, $S(t,\tau) \mathbb{A}(\tau)
= \mathbb{A}(t)$,
for any $t\geqslant \tau$), $\mathbb{A}(t)$ is compact for all $t\in \mathbb{R}$, and pullback attracts
bounded subsets of $X$ at time $t$, for each $t\in \mathbb{R}$.
\end{definition}

In applications, to prove that a process has a pullback attractor
we use the Theorem \ref{theorem pullback}, proved in \cite{CCLR},
which gives a sufficient condition for existence of a compact
pullback attractor. For this, we will need the concept of pullback
strongly bounded dissipativeness.

\begin{definition} 
An evolution process $\{S(t,\tau):t\geqslant \tau \in \mathbb{R}\}$
in $X$ is \emph{pullback strongly bounded dissipative} if, for
each $t\in \mathbb{R}$, there is a bounded subset $B(t)$ of $X$
which pullback absorbs bounded subsets of $X$ at time $s$ for
each $s\leqslant t$; that is, given a bounded subset $B$ of $X$
and $s\leqslant t$, there exists $\tau_0(s,B)$ such that
$S(s,\tau)B \subset B(t)$, for all $\tau\leqslant \tau_0(s,B)$.
\end{definition}

Now we can present the result which guarantees the existence
of pullback attractors for non-\-autonomous problems, see \cite{CCLR}.

\begin{theorem}\label{theorem pullback}
If an evolution process $\{S(t,\tau): t\geqslant \tau\in {\mathbb R}\}$
 in the metric space $X$ is pullback strongly bounded dissipative
and pullback asymptotically compact, then $\{S(t,\tau):t\geqslant
\tau\in {\mathbb R}\}$ has a pullback attractor
$\{\mathbb{A}(t): t\in \mathbb{R}\}$ with the property that
$\underset{\tau \leqslant t } \cup \mathbb{A}(\tau)$ is bounded for
each $t\in \mathbb{R}$.
\end{theorem}

The next result gives sufficient conditions for pullback
asymptotic compactness, and its proof can be found in \cite{CCLR}.

\begin{theorem}\label{pull-asym-comp}
Let $\{S(t,s): t\geqslant s \in {\mathbb R}\}$ be a pullback strongly bounded
evolution process such that $S(t,s) = L(t,s) + U(t,s)$,
where there exist a non-increasing
function $k: \mathbb{R}^{+} \times \mathbb{R}^{+} \to \mathbb{R}$,
with $k(\sigma,r)\to 0$ when $\sigma \to \infty$, and for all
$s\leqslant t$ and $x\in X$ with $\|x\| \leqslant r$,
$\|L(t,s)x\| \leqslant k(t-s,r)$, and $U(t,s)$ is compact. Then, the family of evolution
process $\{S(t,s): t\geqslant s \in {\mathbb R}\}$ is pullback asymptotically compact.
\end{theorem}

\section{Linear Analysis}

In this section we consider the linear problem associated with \eqref{prob1}-\eqref{prob3}, in this case we consider the singularly non-autonomous linear parabolic problem
\begin{equation*}\label{Def-Prob-Lin}
\begin{cases}
w_t=A_{(a)}(t)w,\ t>\tau,\\
w(\tau)=w_0,\ \tau\in\mathbb{R},
\end{cases}
\end{equation*}
where $w$, $w_0$ are defined in \eqref{Def1} and the linear unbounded operador $A_{(a)}$ is defined by \eqref{Def3}-\eqref{Def4}.

It is not difficult to see that $\det(A_{(a)}(t))=\eta\kappa\Lambda^{\frac{3}{2}}$ and $0\in\rho(A_{(a)}(t))$ for any $t\in\mathbb{R}$. Moreover we have that the operator $A^{-1}_{(a)}(t): D(A^{-1}_{(a)}(t))\to Y$ is defined by
\begin{equation*}\label{Defff_1}
D(A^{-1}_{(a)}(t))=L^2(\Omega)\times H^{-2}(\Omega)\times H^{-2}(\Omega)
\end{equation*}
and
\begin{equation*}\label{Defff_2}
\begin{split}
A^{-1}_{(a)}(t)\begin{bmatrix} u\\ v \\ \theta \end{bmatrix}&=\begin{bmatrix} 
\dfrac{1}{\eta\kappa}(a(t))^2\Lambda^{-\frac{1}{2}} & -\dfrac{1}{\eta}\Lambda^{-1}  & -\dfrac{1}{\eta\kappa}a(t)\Lambda^{-1}\\ \vspace{-0.3cm} \\
I & 0 & 0\\ \vspace{-0.3cm} \\
-\dfrac{1}{\kappa} a(t)I & 0 & \dfrac{1}{\kappa}\Lambda^{-\frac{1}{2}} \end{bmatrix}\begin{bmatrix} u\\ v \\ \theta \end{bmatrix}\\
&=\begin{bmatrix} 
\dfrac{1}{\eta\kappa}(a(t))^2\Lambda^{-\frac{1}{2}} u -\dfrac{1}{\eta}\Lambda^{-1}  v -\dfrac{1}{\eta\kappa}a(t)\Lambda^{-1} \theta\\  \vspace{-0.3cm} \\
u \\ \vspace{-0.3cm} \\
 -\dfrac{1}{\kappa}a(t)u+\dfrac{1}{\kappa}\Lambda^{-\frac{1}{2}}\theta \end{bmatrix},
\end{split}
\end{equation*}
for all $t\in\mathbb{R}$.

\begin{proposition}
Denote by $\mathcal{H}$ the extrapolation space of $Y=H^2(\Omega)\times L^{2}(\Omega)\times L^{2}(\Omega)$ generated by operator $A_{a}^{-1}(t)$. The following equality holds
\[
\mathcal{H}=L^2(\Omega)\times H^{-2}(\Omega)\times H^{-2}(\Omega).
\]
\end{proposition}
\proof Recall first that $\mathcal{H}$ is the completion of the normed space $(\mathcal{H},\|A_{a}^{-1}(t)\cdot\|_{\mathcal{H}})$. Since there are positive constants $C_1$ and $C_2$ such that for any $( u, v , \theta )\in L^2(\Omega)\times H^{-2}(\Omega)\times H^{-2}(\Omega)$ we have that
\begin{equation*}
\left\| A_{a}^{-1}(t)\begin{bmatrix} u\\ v \\ \theta \end{bmatrix}\right\|_{H^2(\Omega)\times L^{2}(\Omega)\times L^{2}(\Omega)}\leqslant C_1\left\| \begin{bmatrix} u\\ v \\ \theta \end{bmatrix}\right\|_{L^2(\Omega)\times H^{-2}(\Omega)\times H^{-2}(\Omega)},
\end{equation*}
and
\begin{equation*}
\left\| \begin{bmatrix} u\\ v \\ \theta \end{bmatrix}\right\|_{L^2(\Omega)\times H^{-2}(\Omega)\times H^{-2}(\Omega)}\leqslant C_2\left\| A_{a}^{-1}(t)\begin{bmatrix} u\\ v \\ \theta \end{bmatrix}\right\|_{H^2(\Omega)\times L^{2}(\Omega)\times L^{2}(\Omega)}.
\end{equation*}

Below is the proof of this last statement. Let $[u\ \ v\ \ \theta]^T\in L^2(\Omega)\times H^{-2}(\Omega)\times H^{-2}(\Omega)$, and note that

\begin{equation*}
\begin{split}
&\left\|A^{-1}_{(a)}(t)\begin{bmatrix} u \\ v \\ \theta\end{bmatrix}\right\|_{Y} =\left\|\begin{bmatrix} 
\dfrac{1}{\eta\kappa}(a(t))^2\Lambda^{-\frac{1}{2}} u -\dfrac{1}{\eta}\Lambda^{-1}  v -\dfrac{1}{\eta\kappa}a(t)\Lambda^{-1} \theta\\ \vspace{-0.3cm} \\
u \\ \vspace{-0.3cm} \\
 -\dfrac{1}{\kappa}a(t)u+\dfrac{1}{\kappa}\Lambda^{-\frac{1}{2}}\theta \end{bmatrix}\right\|_{H^2(\Omega)\times L^{2}(\Omega)\times L^{2}(\Omega)}
\\
&=\left\|\dfrac{1}{\eta\kappa}(a(t))^2\Lambda^{-\frac{1}{2}} u -\dfrac{1}{\eta}\Lambda^{-1}  v -\dfrac{1}{\eta\kappa}a(t)\Lambda^{-1} \theta \right\|_{H^2(\Omega)} + \|u\|_{L^2(\Omega)}+\left\|-\dfrac{1}{\kappa}a(t)u+\dfrac{1}{\kappa}\Lambda^{-\frac{1}{2}}\theta \right\|_{L^2(\Omega)}\\
&\leqslant \left\|\dfrac{1}{\eta\kappa}(a(t))^2\Lambda^{-\frac{1}{2}} u\right\|_{H^2(\Omega)} +\left\|\dfrac{1}{\eta}\Lambda^{-1}  v \right\|_{H^2(\Omega)} + \left\|\dfrac{1}{\eta\kappa}a(t)\Lambda^{-1} \theta \right\|_{H^2(\Omega)}  \\
& \quad +\|u\|_{L^ 2(\Omega)}+ \left\|-\dfrac{1}{\kappa}a(t)u\right\|_{L^2(\Omega)}+ \left\|\dfrac{1}{\kappa}\Lambda^{-\frac{1}{2}}\theta \right\|_{L^2(\Omega)}\\
&\leqslant \dfrac{a_1^2}{\eta\kappa}\| \Lambda^{-\frac{1}{2}}u\|_{H^2(\Omega)}+ \dfrac{1}{\eta}\|\Lambda^{-1}v\|_{H^2(\Omega)}+
\dfrac{a_1}{\eta\kappa} \|\Lambda^{-1}\theta\|_{H^2(\Omega)} + \left(1+\dfrac{a_1}{\kappa}\right)\|u\|_{L^ 2(\Omega)} +\dfrac{1}{\kappa}\|\Lambda^{-\frac{1}{2}}\theta\|_{L^2(\Omega)}\\
&=\left(1+\dfrac{a_1}{\kappa}+\dfrac{a_1^ 2}{\eta \kappa}\right)\|u\|_{L^ 2(\Omega)} +\dfrac{1}{\eta}\|v\|_{H^{-2}(\Omega)} + \left(\dfrac{a_1}{\eta \kappa}+\dfrac{1}{\kappa}\right)\|\theta\|_{H^{-2}(\Omega)}\\
&\leqslant C_1\left\| \begin{bmatrix} u\\ v \\ \theta \end{bmatrix}\right\|_{L^2(\Omega)\times H^{-2}(\Omega)\times H^{-2}(\Omega)}.
\end{split}
\end{equation*}

On the other hand, let $[u\ \ v\ \ \theta]^T\in L^2(\Omega)\times H^{-2}(\Omega)\times H^{-2}(\Omega)$, and note firstly that 
\begin{equation}\label{lmad}
\begin{split}
\left\| \begin{bmatrix} u\\ v \\ \theta \end{bmatrix}\right\|_{L^2(\Omega)\times H^{-2}(\Omega)\times H^{-2}(\Omega)}&=\|u\|_{L^2(\Omega)}+\|v\|_{H^{-2}(\Omega)}+\|\theta\|_{H^{-2}(\Omega)}\\
&=\|u\|_{L^2(\Omega)}+\|\Lambda^{-\frac{1}{2}}v\|_{L^{2}(\Omega)}+\|\Lambda^{-\frac{1}{2}}\theta\|_{L^{2}(\Omega)}.
\end{split}
\end{equation}

The last two parcels of \eqref{lmad} can be estimated as follows 
\begin{equation}\label{lmadk}
\begin{split}
\|\Lambda^{-\frac{1}{2}}\theta\|_{L^{2}(\Omega)}& = \kappa\left\| \frac{1}{\kappa}\Lambda^{-\frac{1}{2}}\theta-\dfrac{1}{\kappa}a(t)u+\dfrac{1}{\kappa}a(t)u\right\|_{L^{2}(\Omega)}\\
&\leqslant \left\|\Lambda^{-\frac{1}{2}}\theta-\dfrac{1}{\kappa}a(t)u\right\|_{L^{2}(\Omega)}+ \|a(t)u\|_{L^{2}(\Omega)}\\
&\leqslant \kappa \left\| \frac{1}{\kappa}\Lambda^{-\frac{1}{2}}\theta-\dfrac{1}{\kappa}a(t)u\right\|_{L^{2}(\Omega)}+ a_1\|u\|_{L^{2}(\Omega)}
\end{split}
\end{equation}
and
\begin{equation}\label{lmadf}
\begin{split}
&\|\Lambda^{-\frac{1}{2}}v\|_{L^{2}(\Omega)} = {\eta} \left\|-\frac{1}{\eta}\Lambda^{-\frac{1}{2}}v \right\|_{L^{2}(\Omega)} \\
&\leqslant \eta \left\|-\frac{1}{\eta} \Lambda^{-\frac{1}{2}} v- \frac{1}{\eta \kappa}a(t)\Lambda^{-\frac{1}{2}}\theta + \frac{1}{\eta \kappa}a^2(t)u \right\|_{L^{2}(\Omega)} + \left\|\frac{1}{\kappa}a(t)\Lambda^{-\frac{1}{2}}\theta \right\|_{L^{2}(\Omega)}+ \left\|\frac{1}{\kappa}a^2(t)u \right\|_{L^{2}(\Omega)}\\
&\leqslant \eta \left\|\Lambda^{\frac{1}{2}}\Big(\frac{1}{\eta \kappa}a^2(t)\Lambda^{-\frac{1}{2}}u- \frac{1}{\eta}\Lambda^{-1}v - \frac{1}{\eta\kappa}a(t) \Lambda^{-1}\theta\Big) \right\|_{L^{2}(\Omega)} + a_1 \left\|\frac{1}{\kappa} \Lambda^{-\frac{1}{2}}\theta \right\|_{L^{2}(\Omega)}+ \frac{1}{\kappa}a_1^2\|u\|_{L^{2}(\Omega)}\\
&\leqslant \eta \left\| \frac{1}{\eta \kappa}a^2(t)\Lambda^{-\frac{1}{2}}u- \frac{1}{\eta}\Lambda^{-1}v - \frac{1}{\eta\kappa}a(t) \Lambda^{-1}\theta\right\|_{H^{2}(\Omega)} + a_1\left\|\frac{1}{\kappa} \Lambda^{-\frac{1}{2}}\theta\right\|_{L^{2}(\Omega)}+ \frac{1}{\kappa}a_1^2\|u\|_{L^{2}(\Omega)}\\
&\leqslant \eta\left\| \frac{1}{\eta \kappa}a^2(t)\Lambda^{-\frac{1}{2}}u- \frac{1}{\eta}\Lambda^{-1}v - \frac{1}{\eta\kappa}a(t) \Lambda^{-1}\theta \right\|_{H^{2}(\Omega)} + a_1\left\| \frac{1}{\kappa}\Lambda^{-\frac{1}{2}}\theta-\dfrac{1}{\kappa}a(t)u\right\|_{L^{2}(\Omega)}\\
&\quad+ \frac{2}{\kappa}a_1^2\|u\|_{L^{2}(\Omega)}.
\end{split}
\end{equation}
Then, combining \eqref{lmad} with \eqref{lmadk} and \eqref{lmadf} we obtain that
\begin{equation*}
\begin{split}
\left\| \begin{bmatrix} u\\ v \\ \theta \end{bmatrix}\right\|_{L^2(\Omega)\times H^{-2}(\Omega)\times H^{-2}(\Omega)}\leqslant C_2\left\| A_{a}^{-1}(t)\begin{bmatrix} u\\ v \\ \theta \end{bmatrix}\right\|_{H^2(\Omega)\times L^{2}(\Omega)\times L^{2}(\Omega)}.
\end{split}
\end{equation*}

So we conclude that the completion of $(\mathcal{H},\|A_{a}^{-1}(t)\cdot\|_{\mathcal{H}})$ and $(\mathcal{H},\|\cdot\|_{L^2(\Omega)\times H^{-2}(\Omega)\times H^{-2}(\Omega)})$ coincide.
\qed

Note that the operator $A_{a}(t)$ can be  extended to its closed $\mathcal{H}-$realization (see Amann \cite{A} p. 262), which we will still denote by the same symbol  so that $A_{a}(t)$ considered in $\mathcal{H}$ is then sectorial positive operator (see [5]). Our next concern will be to obtain embedding of the spaces from the fractional powers scale $\mathcal{H}^\alpha$, $\alpha\geqslant 0$, generated by $(A_{a}(t), \mathcal{H})$.

\begin{theorem}
The operators $A_{(a)}(t)$ are uniformly sectorial and the map $\mathbb{R}\ni t\mapsto A_{(a)}(t)\in \mathcal{L}(Y, \mathcal{H})$ is uniformly H\"{o}lder continuous. Then, for each functional parameter $a$, there exist a process 
\[
\{U_{(a)}(t,\tau): t\geqslant \tau\in \mathbb{R}\}
\] 
(or simply $L_{(a)}(t,\tau)$) associated with the operator $A_{(a)}(t)$, that is given by
\[
L_{(a)}(t,\tau) = e^{-(t-\tau)A_{(a)}(\tau)} + \int_{\tau}^{t}L_{(a)}(t,s)[A_{(a)}(\tau) -
A_{(a)}(s)]e^{-(s-\tau)A_{(a)}(\tau)}ds,\ \forall t\geqslant\tau.
\]
The linear evolution operator $\{L_{(a)}(t,\tau): t\geqslant \tau \in \mathbb{R}\}$ satisfies the condition \eqref{1.65}. 
\end{theorem}

\begin{proof}
From \cite{LT} it follows that $A_{(a)}(t)$ is uniformly sectorial (in $Y$); that is there is a constant $M>0$ (independent of $t$) such that
\[
\|(\lambda I+A_{(a)}(t))^{-1} \|_{L(Y)}\leqslant\dfrac{M}{|\lambda|+1},\ \mbox{for all}\ \lambda\in\mathbb{C}\ \mbox{with}\ \mbox{Re}\lambda\geqslant0.
\]

Now, note that for  $[u\ \ v \ \ \theta]^T\in H^2(\Omega)\times L^2(\Omega)\times L^2(\Omega)$, and $t,s\in\mathbb{R}$, we can estimate the norm $\|[(A_{(a)}(t)-A_{(a)}(s))[u\ \ v \ \ \theta]^T\|_{ L^2(\Omega)\times H^{-2}(\Omega)\times H^{-2}(\Omega)}$ using \eqref{eta-c2} in the following way
\[
\left\|(A_{(a)}(t)-A_{(a)}(s)) \begin{bmatrix} u\\ v \\ \theta \end{bmatrix}\right\|_{L^2(\Omega)\times H^{-2}(\Omega)\times H^{-2}(\Omega)}
\]
\begin{equation*}
\begin{split}
&=\left\|\begin{bmatrix}0 & 0 & 0 \\ 0 & 0 & -(a(t)-a(s))\Lambda^{\frac12}\\ 0 & (a(t)-a(s))\Lambda^{\frac12} & 0 \end{bmatrix}\begin{bmatrix} u\\ v \\ \theta \end{bmatrix}\right\|_{L^2(\Omega)\times H^{-2}(\Omega)\times H^{-2}(\Omega)}\\
&=\left\|\begin{bmatrix}0 \\ -(a(t)-a(s))\Lambda^{\frac12}\theta \\ (a(t)-a(s))\Lambda^{\frac12}v\end{bmatrix}\right\|_{L^2(\Omega)\times H^{-2}(\Omega)\times H^{-2}(\Omega)}\\
&= |a(t)-a(s)|\|(-\Delta)\theta\|_{H^{-2}(\Omega)}+|a(t)-a(s)|\|(-\Delta) v\|_{H^{-2}(\Omega)}\\
&=|a(t)-a(s)|\|\theta\|_{L^2(\Omega)}+|a(t)-a(s)|\|v\|_{L^2(\Omega)}\\
&=|a(t)-a(s)|(\|\theta\|_{L^2(\Omega)}+\|v\|_{L^2(\Omega)})\\
&\leqslant c|t-s|^{\beta}\left\|\begin{bmatrix} u\\ v \\ \theta \end{bmatrix}\right\|_{H^2(\Omega)\times L^2(\Omega)\times L^2(\Omega)},
\end{split}
\end{equation*}
for any $t,\tau,s\in\mathbb{R}$, hence the map $\mathbb{R}\ni t\mapsto A_{(a)}(t)\in \mathcal{L}(Y)$ is uniformly H\"older continuous and 
\begin{equation*}
\begin{split}
\|A_{(a)}(t)-A_{(a)}(s)\|_{\mathcal{L}(H^2(\Omega)\times L^2(\Omega)\times L^2(\Omega), L^2(\Omega)\times H^{-2}(\Omega)\times H^{-2}(\Omega))}\leqslant c(t-s)^{\beta}.
\end{split}
\end{equation*}
\qed

\medskip

Therefore, there exists a linear evolution process $\{L_{(a)}(t,\tau): t\geqslant \tau\in \mathbb{R}\}$ associated with the operator $A_{(a)}(t)$, that is given by
\[
L_{(a)}(t,\tau) = e^{-(t-\tau)A_{(a)}(\tau)} + \int_{\tau}^{t}L_{(a)}(t,s)[A_{(a)}(\tau) -
A_{(a)}(s)]e^{-(s-\tau)A_{(a)}(\tau)}ds,\ \forall t\geqslant\tau.
\]
Futhermore, the process $\{L_{(a)}(t,\tau): t\geqslant \tau \in \mathbb{R}\}$ satisfies the condition \eqref{1.65}.
\end{proof}

\section{Existence of global solutions}\label{Sec-Exist} 

In this section we study the existence of global solutions for \eqref{Def-Prob}. It is not difficult to prove the following result, see for instance Lemma 2.4 in \cite{Karina}. 

\begin{lemma}\label{lem:fregul}
Let $f \in C^1(\mathbb{R}^2)$ be a function such that the condition \eqref{Cond_f2} is holds. Then
$$
| f(t,s_1)-f(t,s_2) |  \leqslant  2^{\rho-1}c  | s_1-s_2 |
\big(1 +| s_1 |^{\rho-1} + | s_2 |^{\rho-1}  \big),\ 
\forall\ t,s_1,s_2 \in \mathbb{R}.
$$
\end{lemma}

The next result also can be found in \cite{Karina}, see Lemma 2.5, we present the proof for the sake of completeness. 

\begin{lemma}\label{lem:fewelldef}
Assume that $1<\rho < \frac{N+4}{N-4}$ and let $f \in C^1(\mathbb{R}^2)$ be
a function such that the condition \eqref{Cond_f2} holds. Then
there exists $\alpha \in (0,1)$ such that the
Nemitski\u{\i} operator $f^e(t,\cdot): X^{\frac12}\to
X^{-\frac{\alpha}{2}}$ is Lipschitz continuous in bounded subsets
of $ X^{\frac12}$, uniformly in $t\in\mathbb{R}$.
\end{lemma}

\proof Let  be $\alpha \in (0,1)$ such that
\begin{equation}\label{eq:alpha}
\rho \leqslant  \frac{N+4 \alpha}{N-4}.
\end{equation}
 Since $X^{\gamma} \hookrightarrow H^{{4 \gamma}}(\Omega)$,
we have $X^{\frac12} \hookrightarrow X^{\frac{\alpha}{2}}
\hookrightarrow H^{{2 \alpha}}(\Omega) \hookrightarrow
L^{\frac{2N}{(N-4\alpha)}}(\Omega)$.
Therefore $L^{\frac{2N}{(N-4\alpha)}}(\Omega) \hookrightarrow
X^{-\frac{\alpha}{2}}$. Now by Lemma \ref{lem:fregul} and H\"older's
Inequality we obtain
\begin{align*}
&\|f^e(t,u) - f^e(t,v)\|_{X^{-\frac{\alpha}{2}}}\\
& \leqslant \tilde{c}\,  \|f^e(t,u) - f^e(t,v)
\|_{L^{\frac{2N}{(N-4\alpha)}}(\Omega)} \\
&  \leqslant \tilde{c}\, \Big( \int_{\Omega}
[ 2^{\rho-1}c \, |u-v|(1 + |u|^{\rho-1} +  |v|^{\rho-1})
]^{\frac{2N}{(N-4\alpha)}} \Big)^{\frac{N+4 \alpha}{2N}} \\
&  \leqslant \tilde{\tilde{c}}\,  \|u-v\|_{L^{\frac{2N}{N-4 \alpha}}
 (\Omega)}  \Big( \int_{\Omega} \big( 1 +  |u|^{\rho-1} +  |v|^{\rho-1}
 \big)^{\frac{N}{4 \alpha}} \Big)^{\frac{4 \alpha}{N}} \\
&  \leqslant \tilde{\tilde{\tilde{c}}}\,
\|u-v\|_{L^{\frac{2N}{N-4 \alpha}} (\Omega)}
\Big(1 + \|u\|_{L^{\frac{N(\rho-1)}{4 \alpha}}(\Omega)}^{\rho-1}
 + \|v\|_{L^{\frac{N(\rho-1)}{4 \alpha}}(\Omega)}^{\rho-1} \Big),
\end{align*}
where ${\tilde{c}}$ is the embedding constant from
$L^{\frac{2N}{N+4 \alpha}}(\Omega)$ to $X^{-\frac{\alpha}{2}}$.

From Sobolev embeddings $X^{\frac12} \hookrightarrow
X^{\frac{\alpha}{2}} \hookrightarrow H^{2 \alpha}(\Omega)
\hookrightarrow L^{\frac{N(\rho-1)}{4 \alpha}}(\Omega)$  for all
$1< \rho \leqslant \frac{N+4 \alpha}{N-4}$, it follows that for any $t\in\mathbb{R}$
\[
\|f^e(t,u) - f^e(t,v)\|_{X^{-\frac{\alpha}{2}}} \leqslant C_1
\|u-v\|_{X^{\frac12}} \big(1 +  \|u\|_{X^{\frac12}}^{\rho-1} +
\|v\|_{X^{\frac12}}^{\rho-1} \big),
\]
for some constant $C_1>0$.
\qed

\begin{remark}
Since $L^{\frac{2N}{(N+4)}}(\Omega) \hookrightarrow L^2(\Omega)$, it
follows from the proof of the Lemma \ref{lem:fewelldef} that
$f^e(t,\cdot): X^{\frac12} \to L^2(\Omega)$ is Lipschitz continuous in bounded
subsets, uniformly in $t\in\mathbb{R}$; that is,
\[
\|f^e(t,u) - f^e(t,v)\|_{L^2(\Omega)}  \leqslant  \tilde{c}\,
\|f^e(t,u) - f^e(t,v)\|_{L^{\frac{2N}{(N+4)}}(\Omega)} \leqslant
\tilde{\tilde{c}} \|u-v\|_{X^{\frac12}}.
\]
\end{remark}

\begin{remark}
If $\mathcal{H} =: Y^0$, then the fractional power spaces $Y^\alpha$, $\alpha\in [0,1]$, are given by
\[
Y^\alpha = [Y^1,Y^0]_{\alpha} = X^{\frac{1-\alpha}{2}} \times X^{-\frac{\alpha}{2}} \times X^{-\frac{\alpha}{2}},
\]
where $[\cdot,\cdot]_{\alpha}$ denotes the complex interpolation functor (see \cite{triebel}).
\end{remark}

\begin{corollary}\label{corol:lips}
If $f$ is as in Lemma \ref{lem:fewelldef}  and
$\alpha \in (0,1)$ satisfies \eqref{eq:alpha}, the function
$F: \mathbb{R}\times Y \to Y^{\alpha}$, given by $F(t,\begin{bsmallmatrix} u\\ v\\ \theta\end{bsmallmatrix})=\begin{bsmallmatrix} 0\\ f^e(t,u)\\ 0\end{bsmallmatrix}$ is Lipschitz continuous in bounded
subsets of $ Y$, uniformly in $t\in\mathbb{R}$.
\end{corollary}

Now, Theorem \ref{teo:existunicsol} guarantees local
well posedness for the problem \eqref{Def-Prob} in the energy
space $Y$.

\begin{corollary}
If $f, F$  are like in the Corollary \ref{corol:lips} and
$\alpha \in (0,1)$ satisfies \eqref{eq:alpha}, then given $r > 0$, there is a time
$\tau=\tau(r) > 0$, such that for all
$w_0 \in B_{Y}(0;r)$ there exists a unique solution
$w: [t_0,t_0+{\tau}] \to Y $ of the
problem \eqref{Def-Prob} starting in $w_0$.
Moreover, such solutions are continuous with respect the
initial data in $B_{Y}(0;r)$.
\end{corollary}

Since $\tau$ can be chosen uniformly in bounded
subsets of $Y$, the solutions which do not blow up in $Y$ must exist globally. Alternatively, we obtain a uniform in time estimate of $\|(u(t),\partial_u(t),\theta(t))\|_{Y}$, such estimate is needed to justify global solvability of the problem \eqref{Def-Prob} in $Y=H^2(\Omega)\times L^2(\Omega)\times L^2(\Omega)$.

Consider the original system \eqref{prob1} (or \eqref{Def-Prob} in $Y$). Multiplying the first equation in \eqref{prob1} by $u_t$, and the second equation in \eqref{prob1} by $\theta$, we get the system
\[
\begin{cases}
u_{tt}u_t+\eta\Delta^2 uu_t+a(t)\Delta\theta u_t=f(t,u)u_t,\ t>\tau,\ x\in\Omega,\\
\theta_t\theta-\kappa\Delta \theta\theta-a(t)\Delta u_t\theta=0,\ t>\tau,\ x\in\Omega,
\end{cases}
\]
and integrating over $\Omega$ we obtain
\begin{equation}\label{C1}
\begin{cases}
\dfrac{1}{2}\dfrac{d}{dt}\displaystyle\int_\Omega|u_t|^2dx+\dfrac{\eta}{2}\dfrac{d}{dt}\displaystyle\int_\Omega| \Delta u|^2dx+a(t)\displaystyle\int_\Omega\Delta\theta u_t dx=\dfrac{d}{dt}\displaystyle\int_\Omega\int_0^u f(t,s)dsdx,\\ \\
\dfrac{1}{2}\dfrac{d}{dt}\displaystyle\int_\Omega|\theta|^2dx+ \kappa\displaystyle\int_\Omega|\nabla \theta|^2dx-a(t)\displaystyle\int_\Omega \Delta u_t\theta dx=0,
\end{cases}
\end{equation}
for any $t>\tau$. Note that 
\begin{equation}\label{C2}
\displaystyle\int_\Omega(-\Delta)\theta u_t dx = \displaystyle\int_\Omega \theta(-\Delta)u_t dx.
\end{equation}

Combining \eqref{C1} with \eqref{C2} we have
\begin{equation}\label{C3}
\begin{split}
&\dfrac{d}{dt}\dfrac{1}{2}\left(\displaystyle\int_\Omega|u_t|^2dx+\eta\displaystyle\int_\Omega| \Delta u|^2dx+\displaystyle\int_\Omega|\theta|^2dx-2\displaystyle\int_\Omega\int_0^u f(t,s)dsdx\right)\\
&=-\kappa\displaystyle\int_\Omega|\nabla \theta|^2dx,
\end{split}
\end{equation}
for any $t>\tau$.

The total energy of the system $\mathcal{E}(t)$ associated with the solution $(u(t),\partial_tu(t),\theta(t))$ of \eqref{prob1}-\eqref{prob3} is defined by 
\begin{equation}\label{En-Funct} 
\mathcal{E}(t) = \dfrac{\eta}{2}\|u(t)\|^2_{H^2(\Omega)}+\dfrac{1}{2}\|u_t(t)\|^2_{L^2(\Omega)}+\dfrac{1}{2}\|\theta(t)\|^2_{L^2(\Omega)}-\displaystyle\int_\Omega\int_0^u f(t,s)dsdx.
\end{equation}

This identity says that the function $t\mapsto \mathcal{E}(t)$ becomes monotone decreasing.

\medskip

We obtain (from \eqref{Cond_f1}) that for each $\varepsilon>0$, there exists $C_\varepsilon>0$ such that
\begin{equation}\label{CCd}
\int_\Omega\int_0^{u(\cdot,t)} f(t,s)ds dx\leqslant \varepsilon\|u(\cdot,t)\|^2_{L^2(\Omega)} +C_\varepsilon,
\end{equation}
then the property $\mathcal{E}(t)\leqslant \mathcal{E}(\tau)$ offers an a priori estimate of the solution $(u(t),\partial_tu(t),\theta(t))$ in $H^2(\Omega)\times L^{2}(\Omega)\times L^{2}(\Omega)$. In fact,
\begin{equation*}\label{mla}
\begin{split}
\dfrac{1}{2}\left\| \begin{bmatrix} u\\ v \\ \theta \end{bmatrix}\right\|_{H^2(\Omega)\times L^{2}(\Omega)\times L^{2}(\Omega)}^2 
&\leqslant c\mathcal{E}(\tau)+c\varepsilon_0\|u(\cdot,t)\|^2_{L^2(\Omega)} +C_{\varepsilon_0}\\
&\leqslant c\mathcal{E}(\tau)+c\varepsilon_0\left\| \begin{bmatrix} u\\ v \\ \theta \end{bmatrix}\right\|_{H^2(\Omega)\times L^{2}(\Omega)\times L^{2}(\Omega)}^2 +C_{\varepsilon_0},
\end{split}
\end{equation*}
and, if we choose  $0<\varepsilon_0 < \dfrac{1}{2c}$, we get a boundedness as desired, that is,
\[
\limsup_{t \to +\infty}\left\| \begin{bmatrix} u\\ v \\ \theta \end{bmatrix}\right\|_{H^2(\Omega)\times L^{2}(\Omega)\times L^{2}(\Omega)}
 < +\infty.
\]

With this, we ensure that there exists a global solution $w(t)$ for Cauchy problem \eqref{Def-Prob} in $Y$ and it defines an evolution process  $\{S_{(a)}(t,\tau): t \geqslant \tau \in \mathbb{R}\}$, that is, 
\[
S_{(a)}(t, \tau)w_0 = w(t), \ \forall\, t \geqslant \tau\in\mathbb{R}.
\]
According to \cite{CN}
\begin{equation}\label{eq:evoper}
S_{(a)}(t,\tau)w_0 = L_{(a)}(t,\tau)w_0 + \int_\tau^t L_{(a)}(t,s)F(s,S_{(a)}(s,\tau)w_0) ds,
\quad \forall \, t \geqslant \tau \in \mathbb{R},
\end{equation}
where $\{L_{(a)}(t,\tau): t\geqslant \tau \in \mathbb{R} \}$ is the linear evolution process associated with the homogeneous  problem \eqref{Def-Prob}.

\section{Dissipativeness of the thermoelastic equation}\label{Sec-Diss} 

In this section we combine the arguments from  \cite{ASM}, \cite{AFQ}, \cite{BMTF}, \cite{CCF} and \cite{GRP} in order to prove the existence of pullback attractors for \eqref{prob1}-\eqref{prob3}. To achieve our purpose we consider the functionals
\begin{equation}\label{Def-phi}
\phi(t)= \int_\Omega uu_t dx
\end{equation} 
and
\begin{equation}\label{Def-psi}
\psi(t)=-\int_\Omega u_t(\Delta^{-1}\theta) dx.
\end{equation}

From this we define an energy functional
\begin{equation}\label{Deff}
\mathcal{L}(t)=M\mathcal{E}(t)+\delta_1\phi(t)+\delta_2\psi(t)
\end{equation}
where 
\begin{equation}\label{mald}
0<\delta_1<\delta_2<1\ \mbox{and}\ M>0
\end{equation} 
will be fixed later. We recall that $\mathcal{E}(t)$ is decreasing since $\mathcal{E}'(t)\leqslant0$ from \eqref{C3}.

\begin{theorem}\label{tmfjk}
Assume that $0<\eta\leqslant 2$. For $M>0$ sufficiently large, there exist constants $M_1,M_2>0$ such that
\begin{equation}\label{makdu}
\mathcal{L}'(t)\leqslant - M_1\mathcal{E}(t)+M_2,
\end{equation}
for all $t\geqslant0$.
\end{theorem}

\proof
In the following, $C_0$ and $C_1$ will denote positive constants depending on the embedding constants and initial data, respectively, as far it is necessary.

Note that
\begin{equation}\label{mtf0}
\mathcal{L}'(t)=M\mathcal{E}'(t)+\delta_1\phi'(t)+\delta_2\psi'(t).
\end{equation}

Due to \eqref{Deff} and Poincar\'e's  inequality we have  
\begin{equation}\label{mtf1}
\begin{split}
M\mathcal{E}'(t)&=-\kappa M\displaystyle\int_\Omega|\nabla \theta|^2dx\\
&\leqslant -\dfrac{\kappa M}{2}\displaystyle\int_\Omega|\nabla \theta|^2dx-\dfrac{\kappa \lambda_1M}{2}\int_\Omega|\theta|^2dx,
\end{split}
\end{equation}
where $\lambda_1$ is the first eigenvalue of negative Laplacian operator with zero Dirichlet boundary condition. Furthermore, 
\[
\begin{split}
\delta_1\phi'(t)&=\delta_1\int_\Omega |u_t|^2 dx+\delta_1\int_\Omega uu_{tt}dx\\
&=\delta_1\int_\Omega |u_t|^2dx-\delta_1\eta\int_\Omega|\Delta u|^2dx-a(t)\delta_1\int_\Omega \Delta \theta udx+\delta_1\int_\Omega f(t,u)udx\\
&=\delta_1\int_\Omega |u_t|^2dx-\delta_1\eta\int_\Omega|\Delta u|^2dx-a(t)\delta_1\int_\Omega  \theta \Delta udx+\delta_1\int_\Omega f(t,u)udx
\end{split}
\]
and from \eqref{eta-c1} we get
\[
\delta_1\phi'(t)\leqslant \delta_1\int_\Omega |u_t|^2dx-\delta_1\eta\int_\Omega|\Delta u|^2dx-a_0\delta_1\int_\Omega  \theta\Delta udx+\delta_1\int_\Omega f(t,u)udx
\]
and by Young's inequality
\[
\delta_1\phi'(t)\leqslant \delta_1\int_\Omega |u_t|^2dx-\delta_1\eta\int_\Omega|\Delta u|^2dx+\dfrac{a_0}{2}\int_\Omega | \theta|^2dx+\dfrac{a_0\delta_1^2}{2}\int_\Omega |\Delta u|^2dx+\delta_1\int_\Omega f(t,u)udx.
\]

To deal with the integral term, just notice that from dissipativeness condition \eqref{Cond_f1}, there exists $C_\nu >0$ such that
\[
\int_{\Omega} f(t,u) udx \leqslant \nu \|u\|_{L^2(\Omega)}^2
+ C_{\nu},
\]
and thus,
\begin{equation}\label{mtf2}
\begin{split}
&\delta_1\phi'(t)\\
&\leqslant \delta_1\int_\Omega |u_t|^2dx-\delta_1\eta\int_\Omega|\Delta u|^2dx+\dfrac{a_0}{2}\int_\Omega | \theta|^2dx+\dfrac{a_0\delta_1^2}{2}\int_\Omega |\Delta u|^2dx+\delta_1\nu\int_\Omega|u|^2dx+\delta_1C_{\nu}\\
&\leqslant \delta_1\int_\Omega |u_t|^2dx-\delta_1\eta\int_\Omega|\Delta u|^2dx+\dfrac{a_0}{2}\int_\Omega | \theta|^2dx+\dfrac{a_0\delta_1^2}{2}\int_\Omega |\Delta u|^2dx+\dfrac{\delta_1\nu}{\lambda_1}\int_\Omega|\nabla u|^2dx+\delta_1C_{\nu}\\
&\leqslant \delta_1\int_\Omega |u_t|^2dx-\delta_1\eta\int_\Omega|\Delta u|^2dx+\dfrac{a_0}{2}\int_\Omega | \theta|^2dx+\dfrac{a_0\delta_1^2}{2}\int_\Omega |\Delta u|^2dx+\dfrac{\delta_1\mu_1\nu}{\lambda_1}\int_\Omega|\Delta u|^2dx+\delta_1C_{\nu}\\
&= \delta_1\int_\Omega |u_t|^2dx-\delta_1\left(\eta-\frac{\mu_1\nu}{\lambda_1}\right)\int_\Omega|\Delta u|^2dx+\dfrac{a_0\delta_1^2}{2}\int_\Omega |\Delta u|^2dx+\dfrac{a_0}{2}\int_\Omega | \theta|^2dx+\delta_1C_{\nu},\\
\end{split}
\end{equation}
where $\mu_1>0$ is the embedding constant for $\|\nabla u\|_{L^2(\Omega)}^2\leqslant\mu_1\|\Delta u\|_{L^2(\Omega)}^2$, and $\nu$ is chosen such that $0<\nu<\dfrac{\lambda_1\eta}{\mu_1}$.

We also have
\[
\begin{split}
\delta_2\psi'(t)&=-\delta_2\int_\Omega u_{tt}(\Delta^{-1}\theta) dx-\delta_2\int_\Omega u_t(\Delta^{-1}\theta)_t dx\\
&\leqslant\dfrac{\eta\delta_2^2}{2}\int_\Omega|\nabla u|^2dx+\dfrac{\eta}{2}\int_\Omega|\nabla \theta|^2dx+a(t)\delta_2\int_\Omega|\theta|^2dx-\delta_2\int_\Omega f(t,u)(\Delta^{-1}\theta)dx\\
&-\delta_2\kappa\int_\Omega u_t\theta dx-\delta_2a(t)\int_\Omega|u_t|^2dx
\end{split}
\]
and from \eqref{eta-c1} we get
\[
\begin{split}
\delta_2\psi'(t)&\leqslant \dfrac{\eta\delta_2^2}{2}\int_\Omega|\nabla u|^2dx+\dfrac{\eta}{2}\int_\Omega|\nabla \theta|^2dx+a_1\delta_2\int_\Omega|\theta|^2dx-\delta_2\int_\Omega f(t,u)(\Delta^{-1}\theta)dx\\
&-\delta_2\kappa\int_\Omega u_t\theta dx-\delta_2a_0\int_\Omega|u_t|^2dx.
\end{split}
\]
Since there exists $c_0>0$ such that
\begin{equation*}\label{makditbf}
\int_\Omega|\nabla \varrho|^2dx\leqslant c_0\int_\Omega|\nabla \theta|^2dx,
\end{equation*}
where $\varrho=\Delta^{-1}\theta$, by Young's inequality  and Poincar\'e's inequality we obtain
\[
\begin{split}
&\delta_2\psi'(t) \\ 
&\leqslant  \dfrac{\eta\delta_2^2}{2}\int_\Omega|\nabla u|^2dx+\dfrac{\eta}{2}\int_\Omega|\nabla \theta|^2dx+a_1\delta_2\int_\Omega|\theta|^2dx+\dfrac{\delta_2^2}{2}\int_\Omega |f(t,u)|^2dx +\dfrac{1}{2}\int_\Omega|\Delta^{-1}\theta|^2dx\\
&+\dfrac{\kappa}{2\delta_1}\int_\Omega |\theta|^2 dx +\dfrac{\delta_2\delta_1\kappa}{2}\int_\Omega |u_t|^2 dx-\delta_2a_0\int_\Omega|u_t|^2dx\\
&\leqslant  \dfrac{\eta\delta_2^2}{2}\int_\Omega|\nabla u|^2dx+\dfrac{\eta}{2}\int_\Omega|\nabla \theta|^2dx+a_1\delta_2\int_\Omega|\theta|^2dx+\dfrac{\delta_2^2}{2}\int_\Omega |f(t,u)|^2dx +\dfrac{1}{2\lambda_1}\int_\Omega|\nabla \varrho|^2dx\\
&+\dfrac{\kappa}{2\delta_1}\int_\Omega |\theta|^2 dx+\left(\dfrac{\delta_2\delta_1\kappa}{2}-\delta_2a_0\right)\int_\Omega | u_t|^2 dx.
\end{split}
\]

Hence
\begin{equation}\label{mtf3}
\begin{split}
&\delta_2\psi'(t) \\
&\leqslant \dfrac{\eta\delta_2^2}{2}\int_\Omega|\nabla u|^2dx+\left(\dfrac{\eta}{2}+\dfrac{c_0}{2\lambda_1}\right)\int_\Omega|\nabla \theta|^2dx+a_1\delta_2\int_\Omega|\theta|^2dx+\dfrac{\delta_2^2}{2}\int_\Omega |f(t,u)|^2dx\\
&+\dfrac{\kappa}{2\delta_1}\int_\Omega |\theta|^2 dx+\left(\dfrac{\delta_2\delta_1\kappa}{2}-\delta_2a_0\right)\int_\Omega | u_t|^2 dx\\
&\leqslant \dfrac{\eta\delta_2^2}{2}\int_\Omega|\nabla u|^2dx+\left(\dfrac{\eta}{2}+\dfrac{c_0}{2\lambda_1}\right)\int_\Omega|\nabla \theta|^2dx+a_1\delta_2\int_\Omega|\theta|^2dx+\dfrac{\delta_2^2}{2}\int_\Omega |f(t,u)|^2dx\\
&+\dfrac{\kappa}{2\delta_1}\int_\Omega |\theta|^2 dx+\left(\dfrac{\delta_2\delta_1\kappa}{2}-\delta_2a_0\right)\int_\Omega | u_t|^2 dx.
\end{split}
\end{equation}

From \eqref{Cond_f2}, there exists $C>0$ such that
\[
\int_\Omega |f(t,u)|^2dx\leqslant C\int_\Omega|u|^2dx+C\int_\Omega|u|^{2\rho}dx.
\]
Since the condition $1\leqslant \rho\leqslant\frac{N}{N-4}$ implies that  $H^2(\Omega) \hookrightarrow L^{2\rho}(\Omega)$, we get
\begin{equation}\label{malh}
\int_\Omega |f(t,u)|^2dx \leqslant C\int_\Omega|u|^2dx + \bar{C}\leqslant \bar{C}_1\int_\Omega|\Delta u|^2dx + \bar{C}_2
\end{equation}
whenever $\|u\|_{H^2(\Omega)}\leqslant r$ (as \cite{Karina} and \cite{Carv}).

Let $C_\eta:=\eta-\frac{\mu_1\nu}{\lambda_1}>0$, combining \eqref{mtf0} together with \eqref{mtf1}, \eqref{mtf2}, \eqref{mtf3} and \eqref{malh}, we obtain
\[
\begin{split}
\mathcal{L}'(t)&\leqslant \left(\dfrac{\eta}{2}+\dfrac{c_0}{2\lambda_1}-\dfrac{ M\kappa}{2}\right)\int_\Omega|\nabla \theta|^2dx+\left[\delta_1\left(1+\dfrac{\delta_2\kappa}{2}\right)-\delta_2a_0\right]\int_\Omega | u_t|^2 dx\\
&+\! \left(\dfrac{\mu_1\eta\delta_2^2}{2}+\dfrac{\bar{C}_1\delta_2^2}{2}+\dfrac{a_0\delta_1^2}{2}-C_\eta\delta_1\right) \!\! \int_\Omega|\Delta u|^2dx+\left(\dfrac{a_0}{2}+ \dfrac{\kappa}{2\delta_1}+a_1\delta_2- \dfrac{\kappa\lambda_1M}{2}\right)\int_\Omega|\theta|^2dx\\
&+\dfrac{\delta_2^2\bar{C}_2}{2}+\delta_1C_\nu,
\end{split}
\]
and by \eqref{mald}
\[
\begin{split}
\mathcal{L}'(t)&\leqslant \left(\dfrac{\eta}{2}+\dfrac{c_0}{2\lambda_1}-\dfrac{ M\kappa}{2}\right)\int_\Omega|\nabla \theta|^2dx+\left[\delta_1\left(1+\dfrac{\delta_2\kappa}{2}\right)-\delta_2a_0\right]\int_\Omega | u_t|^2 dx\\
&+\left(\tilde{C}_0\delta_2^2-C_\eta\delta_1\right)\int_\Omega|\Delta u|^2dx+\left(\dfrac{a_0}{2}+ \dfrac{\kappa}{2\delta_1}+a_1\delta_2- \dfrac{\kappa\lambda_1M}{2}\right)\int_\Omega|\theta|^2dx\\
&+\dfrac{\delta_2^2\bar{C}_2}{2}+\delta_1C_\nu,
\end{split}
\]
where $\tilde{C}_0=\frac{\mu_1\eta}{2}+\frac{\bar{C}_1 }{2}+\frac{a_0}{2}>0$.

Let $C_\kappa=1+\frac{\kappa}{2}>0$. Now, fixed $0<\delta_2<1$, choose $\delta_1$ such that
\[
0<\dfrac{\tilde{C}_0}{C_\eta}\delta_2^2<\delta_1<\dfrac{a_0}{C_\kappa}\delta_2,
\]
and, thus
\[
\tilde{C}_0\delta_2^2-C_\eta\delta_1<0\quad \mbox{and}\quad \left(1+\dfrac{\delta_2\kappa}{2}\right)\delta_1<C_\kappa \delta_1<\delta_2a_0.
\]

Finally, choose $M$ sufficient large such that 
\[
\dfrac{\eta}{2}+\dfrac{c_0}{2\lambda_1}-\dfrac{ M\kappa}{2}<0\quad  \mbox{and}\quad \dfrac{a_0}{2}+ \dfrac{\kappa}{2\delta_1}+a_1\delta_2- \dfrac{\kappa\lambda_1M}{2}<0.
\]

With these choices, we will have
\[
\begin{split}
\mathcal{L}'(t)&\leqslant \left[\delta_1\left(1+\dfrac{\delta_2\kappa}{2}\right)-\delta_2a_0\right]\int_\Omega | u_t|^2 dx+\left(\tilde{C}_0\delta_2^2-C_\eta\delta_1\right)\int_\Omega|\Delta u|^2dx\\
&+\left(\dfrac{a_0\delta_1^2}{2}+ \dfrac{\kappa}{2\delta_1}+a_1\delta_2- \dfrac{\kappa\lambda_1M}{2}\right)\int_\Omega|\theta|^2dx+\dfrac{\delta_2^2\bar{C}_2}{2}+\delta_1C_\nu.
\end{split}
\]

Thus
\begin{equation}\label{mkaq}
\mathcal{L}'(t)\leqslant -\bar{M}_1\left(\dfrac{\eta}{2}\|u(t)\|^2_{H^2(\Omega)}+\dfrac{1}{2}\|u_t(t)\|^2_{L^2(\Omega)}+\frac{1}{2}\|\theta(t)\|^2_{L^2(\Omega)}\right)+M_2,
\end{equation}
where 
\[
\bar{M}_1=\min\left\{2\delta_1+\delta_2\kappa-2\delta_2a_0, \dfrac{2\tilde{C}_0\delta_2^2}{\eta}-\dfrac{2C_\eta\delta_1 }{\eta},  a_0\delta_1^2+ \dfrac{\kappa}{\delta_1}+2a_1\delta_2- \kappa\lambda_1M  \right\}>0,
\]
and $M_2=\frac{\delta_2^2\bar{C}_2}{2}+\delta_1C_\nu$.

Now we observe that if $u \in H^2(\Omega) \hookrightarrow
L^{2N/(N-4)}(\Omega) $, then
$$
|u|^{\rho+1} \in L^{2N/\big((N-4)(\rho+1)\big)}(\Omega)
\hookrightarrow L^1(\Omega)
$$
for all $1 < \rho < \frac{N}{N-4}$, and our hypothesis on $f$
implies that $|f(t,s)| \leqslant  c(1 + |s|^{\rho})$,
$s\in \mathbb{R}$.
Therefore, we can find a constant $\bar{c}>1$ such that for all
$u \in X^{1/2}$,
\[
-\int_{\Omega}\int_0^{u} f(t,s)dsdx \leqslant \bar{c} \| u\|_{1/2}^2 (1+ \| u\|_{1/2}^{\rho -1}),
\]
and therefore
\begin{equation}\label{S-ideia}
-\bar{d}\int_{\Omega}\int_0^{u}f(t,s) ds dx
\leqslant \| u\|_{H^2(\Omega)}^2,
\end{equation}
whenever $\| u\|_{1/2} \leqslant r$ and $ \bar{d}= \frac{1}{\bar{c}(1+ r^{\rho-1})}<1$.

Thanks to \eqref{mkaq} and \eqref{S-ideia} we deduce that (since $\bar{d} <1$)
\[
\begin{split}
\mathcal{L}'(t)&\leqslant -\dfrac{\bar{M}_1}{2}\left(\dfrac{\eta}{2}\|u(t)\|^2_{H^2(\Omega)}+\dfrac{1}{2}\|u_t(t)\|^2_{L^2(\Omega)}+\frac{1}{2}\|\theta(t)\|^2_{L^2(\Omega)}\right) -\dfrac{\bar{M}_1}{2}\dfrac{\eta}{2}\|u(t)\|^2_{H^2(\Omega)}+M_2\\
&\leqslant -\dfrac{\bar{M}_1}{2}\left(\dfrac{\eta}{2}\|u(t)\|^2_{H^2(\Omega)}+\dfrac{1}{2}\|u_t(t)\|^2_{L^2(\Omega)}+\frac{1}{2}\|\theta(t)\|^2_{L^2(\Omega)}\right) +\dfrac{\bar{M}_1}{2}\dfrac{\eta}{2}\bar{d}\int_\Omega\int_0^uf(s)ds+M_2\\
&\leqslant -\dfrac{\bar{M}_1\bar{d}}{2}\left(\dfrac{\eta}{2}\|u(t)\|^2_{H^2(\Omega)}+\dfrac{1}{2}\|u_t(t)\|^2_{L^2(\Omega)}+\frac{1}{2}\|\theta(t)\|^2_{L^2(\Omega)} -\dfrac{\eta}{2}\int_\Omega\int_0^uf(s)ds\right)+M_2.
\end{split}
\]

Since $0<\eta\leqslant2$, from \eqref{En-Funct} we conclude 
\[
\mathcal{L}'(t)\leqslant -M_1\mathcal{E}(t)+M_2,
\]
where $M_1=\dfrac{\bar{M}_1\eta\bar{d}}{4}>0$, for all $t\geqslant0$. This concludes the proof of the theorem.
\qed

\begin{remark}\label{Rem_imp}
For every $t\in\mathbb{R}$, from \eqref{CCd} we have
\[
\begin{split}
\mathcal{E}(t) &= \dfrac{\eta}{2}\|u(t)\|^2_{H^2(\Omega)}+\dfrac{1}{2}\|u_t(t)\|^2_{L^2(\Omega)}+\dfrac{1}{2}\|\theta(t)\|^2_{L^2(\Omega)}-\displaystyle\int_\Omega\int_0^u f(t,s)dsdx\\
&\geqslant \dfrac{\eta}{2}\|u(t)\|^2_{H^2(\Omega)}+\dfrac{1}{2}\|u_t(t)\|^2_{L^2(\Omega)}+\dfrac{1}{2}\|\theta(t)\|^2_{L^2(\Omega)}-\varepsilon\|u(t)\|_{L^2(\Omega)}^2-C_\varepsilon\\
&\geqslant \left(\dfrac{\eta}{2}-\dfrac{\varepsilon C_0}{2}\right)\|u(t)\|^2_{H^2(\Omega)}+\dfrac{1}{2}\|u_t(t)\|^2_{L^2(\Omega)}+\dfrac{1}{2}\|\theta(t)\|^2_{L^2(\Omega)}-C_\varepsilon
\end{split}
\]
where $\varepsilon$ is such that $\varepsilon<\dfrac{\eta}{C_0}$, that is
\[
\|\Delta u(t)\|_{L^2(\Omega)}^2+\|u_t(t)\|_{L^2(\Omega)}^2+\|\theta(t)\|^2_{L^2(\Omega)}\leqslant C_1\mathcal{E}(t)+C'_\varepsilon,
\]
where $C_1^{-1}=\min\left\{\left(\dfrac{\eta}{2}-\dfrac{\varepsilon C_0}{2}\right),\dfrac{1}{2}\right\}$.
\end{remark}

\begin{theorem}\label{tmtf}
For $M>0$ sufficiently large, there exist constants $\beta_1,\beta_2,\beta_3,\beta_4>0$ such that
\begin{equation}\label{masxc}
\beta_3\mathcal{E}(t)-\beta_4\leqslant\mathcal{L}(t)\leqslant\beta_1\mathcal{E}(t)+\beta_2, \qquad t\geqslant0.
\end{equation}
\end{theorem}

\proof Note that from Remark \ref{Rem_imp} and \eqref{Def-phi}, there exist $\tilde{C}_1,\tilde{C}_2>0$ such that
\begin{equation}\label{maldd}
\begin{split}
|\phi(t)|&\leqslant\dfrac{1}{2} \|u_t(t)\|_{L^2(\Omega)}^2+\dfrac{\mu_1}{2\lambda_1} \|\Delta u(t)\|_{L^2(\Omega)}^2\\
&\leqslant\max\left\{\dfrac{1}{2},\dfrac{\mu_1}{2\lambda_1}\right\}\left(\|\Delta u(t)\|_{L^2(\Omega)}^2+\|u_t(t)\|_{L^2(\Omega)}^2+\|\theta(t)\|^2_{L^2(\Omega)}\right)\\
&\leqslant \tilde{C}_1\mathcal{E}(t)+\tilde{C}_2.
\end{split}
\end{equation}

Due to Remark \ref{Rem_imp} and \eqref{Def-psi} we also have
\begin{equation}\label{malddla}
\begin{split}
|\psi(t)|&\leqslant C_{\delta_0} \|u_t(t)\|_{L^2(\Omega)}^2+\delta_0 \| \theta(t)\|_{L^2(\Omega)}^2\\
&\leqslant \max\{C_{\delta_0},\delta_0\}\left(\|\Delta u(t)\|_{L^2(\Omega)}^2+\|u_t(t)\|_{L^2(\Omega)}^2+\|\theta(t)\|^2_{L^2(\Omega)}\right)\\
&\leqslant \tilde{C}_3\mathcal{E}(t)+\tilde{C}_4,
\end{split}
\end{equation}
where $C_{\delta_0}>0$.

Now, observe that the constants $\delta_1>0$ and $\delta_2>0$ were fixed in the proof of the Theorem \ref{tmfjk}. Then, combining \eqref{Deff} with \eqref{maldd} and \eqref{malddla} we obtain
\[
\mathcal{L}(t)\leqslant\beta_1\mathcal{E}(t)+\beta_2,\ t\geqslant0.
\]

On the other hand, since
\[
M\mathcal{E}(t)=\mathcal{L}(t)-\delta_1\phi(t)-\delta_2\psi(t)
\]
from \eqref{maldd} and \eqref{malddla},
\[
(M-\delta_1\tilde{C}_1-\delta_2\tilde{C}_3)\mathcal{E}(t)-\delta_1\tilde{C}_2-\delta_2\tilde{C}_4\leqslant  \mathcal{L}(t),
\]
and taking $M>0$ sufficiently large such that $M-\delta_1\tilde{C}_1-\delta_2\tilde{C}_3>0$, we obtain
\[
\beta_3\mathcal{E}(t)-\beta_4\leqslant \mathcal{L}(t).
\]
This concludes the proof of the theorem.
\qed

\begin{corollary}
Under the same conditions of the Theorem \ref{tmfjk} and Theorem \ref{tmtf}, if $B \subset Y$ is a bounded set, and $(u,v,\theta):[\tau, \tau+T] \to Y$, $T>0$, is the solution of \eqref{prob1}-\eqref{prob3}
starting in $(u_0,v_0,\theta_0) \in B$, then there exist positive constants $\bar \omega$, $\gamma_1=\gamma_1(B)$ and $\gamma_2$, such that
\begin{equation}\label{d-e-n}
\|\Delta u(t)\|_{L^2(\Omega)}^2+\|u_t(t)\|_{L^2(\Omega)}^2+\|\theta(t)\|^2_{L^2(\Omega)} \leqslant \gamma_1 e^{-\bar{\omega}(t-\tau)}
+ \gamma_2, \quad t \in [\tau, \tau+T].
\end{equation}
\end{corollary}

\proof From \eqref{makdu} and \eqref{masxc}, we obtain
\[
\mathcal{L}'(t)\leqslant -\sigma_1\mathcal{L}(t)+\sigma_2,
\]
where $\sigma_1=\dfrac{M_1}{\beta_1}$ and $\sigma_2=\dfrac{M_1\beta_2}{\beta_1}+M_2$, and  thus,
\[
\mathcal{L}(t)\leqslant \mathcal{L}(\tau)e^{-\sigma_1( t-\tau)}+\sigma_2 e^{-\sigma_1 t} \int_\tau^t e^{\sigma_1 s} ds\leqslant  \mathcal{L}(\tau)e^{-\sigma_1( t-\tau)}+\dfrac{\sigma_2}{\sigma_1}.
\]

Again, by \eqref{masxc} together with Remark \ref{Rem_imp}, we conclude
\[
\|\Delta u(t)\|_{L^2(\Omega)}^2+\|u_t(t)\|_{L^2(\Omega)}^2+\|\theta(t)\|^2_{L^2(\Omega)} \leqslant \gamma_1 e^{-\sigma_1( t-\tau)}+\gamma_2,
\]
where $\gamma_1=\gamma_1(\mathcal{L}(\tau))>0$ and $\gamma_2>0$.
\qed

\begin{theorem}
Under the same conditions of  Theorem \ref{tmfjk} and Theorem \ref{tmtf}, the problem \eqref{prob1}-\eqref{prob3} has a pullback attractor
$\{\mathcal{A}_{(a)}(t): t\in \mathbb{R}\}$ in $Y$ and
\[
\bigcup_{t\in \mathbb{R}}\mathcal{A}_{(a)}(t)\subset Y.
\]
\end{theorem}

\proof From estimate \eqref{d-e-n} it is easy to check that the evolution process $\{S_{(a)}(t,\tau): t \geqslant \tau \in \mathbb{R}\}$ associated with
\eqref{prob1}-\eqref{prob3} is pullback strongly bounded.

Hence, applying the same ideas of the proofs of the Theorem \ref{tmfjk} and Theorem \ref{tmtf}, we obtain that
the family of evolution process $\{S_{(a)}(t,\tau): t \geqslant \tau \in \mathbb{R}\}$ is pullback asymptotically
compact (see Theorem \ref{pull-asym-comp}). In fact, from \eqref{eq:evoper} we write
\[
S_{(a)}(t,\tau)=L_{(a)}(t,\tau)+U_{(a)}(t,\tau),
\]
where
\[
U_{(a)}(t,\tau):= \int_\tau^t L_{(a)}(t,s)
F(s,S_{(a)}(t,s) ) ds.
\]

With the same arguments used in the proof of the Theorem \ref{tmfjk} with $f\equiv0$ in \eqref{prob1} and with the functionals 
\[
\mathcal{E}(t) = \dfrac{\eta}{2}\|u(t)\|^2_{H^2(\Omega)}+\dfrac{1}{2}\|u_t(t)\|^2_{L^2(\Omega)}+\dfrac{1}{2}\|\theta(t)\|^2_{L^2(\Omega)}
\]
and
\[
\mathcal{L}(t)=M\mathcal{E}(t)+\delta_1\phi(t)+\delta_2\psi(t)
\]
where $\phi$ is defined in \eqref{Def-phi}  and $\psi$ is defined in \eqref{Def-psi}, we get from \eqref{makdu} that there exist $c_1>0$ such that
\[
\mathcal{L}'(t)\leqslant - c_1\mathcal{E}(t)
\]
and from arguments used in the proof of the Theorem \ref{tmtf} with $f\equiv0$  in \eqref{prob1}, by \eqref{masxc} we get $c_2,c_3>0$ such that
\[
c_2\mathcal{E}(t)\leqslant\mathcal{L}(t)\leqslant c_3\mathcal{E}(t)
\]
and hence
\[
\mathcal{E}'(t)\leqslant -c_0 \mathcal{E}(t)
\]
for some $c_0>0$. This ensures that exist constants $K,\alpha>0$ such that
\[
\|L_{(a)}(t,\tau)\|_{\mathcal{L}(Y)}\leqslant K e^{-\alpha(t-\tau)},\ \mbox{for all}\ t\geqslant\tau.
\]

The family of evolution process $\{U_{(a)}(t, \tau) :  t \geqslant \tau\in \mathbb{R}\}$ is compact in $Y$. In fact, the compactness of $U_{(a)}(t,\tau)$ follows easily from the fact that
  $$
 X^{1/2} \stackrel{f^e(t,\cdot)}{\longrightarrow} X^{-\alpha/2}
\hookrightarrow X^{-1/2},
  $$
being the last inclusion compact, since $\alpha <1$ (see Lemma \ref{lem:fewelldef}).

Now, applying Theorem \ref{theorem pullback} we get that the 
problem \eqref{prob1}-\eqref{prob3} has a pullback attractor
$\{\mathcal{A}_{(a)}(t): t\in \mathbb{R}\}$ in $Y$ and that $\bigcup_{t\in \mathbb{R}}\mathcal{A}_{(a)}(t)\subset Y$ is bounded. 
\qed

\section{Acknowledgments}

The authors would like to express their gratitude to Professor Ma To Fu for useful conversations about thermoelasticity.

\end{document}